\documentclass[a4paper,fleqn]{cas-sc}

\usepackage[utf8]{inputenc}
\usepackage{amsthm}
\usepackage{tabularx}
\usepackage{tikz}
\usepackage{capt-of}
\usepackage[numbers,sort&compress]{natbib}

\hypersetup{
  unicode=true,
  hypertexnames=false,
  colorlinks=true,
  linkcolor=blue!55!black,
  citecolor=blue!55!black,
  urlcolor=blue!55!black,
  pdftitle={Nowhere-zero 4-flows in graphs excluding a proper minor of the Petersen graph},
  pdfauthor={József Pintér},
  pdfsubject={Graph theory},
  pdfkeywords={nowhere-zero 4-flow, graph minor, Petersen graph, excluded minor}
}

\newtheorem{theorem}{Theorem}[section]
\newtheorem{lemma}[theorem]{Lemma}
\newtheorem{corollary}[theorem]{Corollary}

\newcommand{\Gm}{\Gamma}
\newcommand{\Ex}{\operatorname{Ex}}
\newcommand{\minor}{\preceq_{\mathrm m}}
\newcommand{\set}[1]{\{#1\}}
\newcommand{\bsep}{;\allowbreak\ }
\newcommand{\pbranch}[1]{\langle #1\rangle}
\newcolumntype{Y}{>{\raggedright\arraybackslash}X}
\newcolumntype{K}[1]{>{\raggedright\arraybackslash}p{#1}}

\begin{document}
\let\WriteBookmarks\relax

\shorttitle{Nowhere-zero 4-flows excluding a proper Petersen minor}
\shortauthors{J. Pint\'er}

\title[mode=title]{Nowhere-zero 4-flows in graphs excluding a proper minor of the Petersen graph}

\author[1,2]{J\'ozsef Pint\'er}
\cormark[1]
\ead{pinterj@edu.bme.hu}

\affiliation[1]{
  organization={Department of Stochastics, Institute of Mathematics,
    Budapest University of Technology and Economics},
  addressline={Egry J\'ozsef utca 1},
  city={Budapest},
  postcode={1111},
  country={Hungary}}

\affiliation[2]{
  organization={HUN-REN--BME Stochastics Research Group},
  addressline={Egry J\'ozsef utca 1},
  city={Budapest},
  postcode={1111},
  country={Hungary}}

\cortext[1]{Corresponding author}

\begin{abstract}
Tutte's $4$-flow conjecture asserts that every finite bridgeless graph with no Petersen minor admits a nowhere-zero $4$-flow. Let $P$ be the Petersen graph and let $e\in E(P)$. We prove that every finite bridgeless $(P/e)$-minor-free multigraph admits a nowhere-zero $4$-flow. Since $P-e$ and $P/e$ are the two maximal proper minors of $P$, combining our result with the theorem of Thomas and Thomson for $(P-e)$-minor-free graphs shows that, for every proper minor $R$ of $P$, every finite bridgeless $R$-minor-free graph admits a nowhere-zero $4$-flow. Equivalently, every finite bridgeless graph without such a flow contains every proper minor of $P$. The proof builds on the girth-five structural framework of Thomas and Thomson together with the nonplanar extension theorem of Norin and Thomas.
\end{abstract}


{\hfuzz=130pt
\maketitle}
\hypersetup{
  pdfauthor={József Pintér},
  pdfsubject={Graph theory}
}

\section{Introduction}

Tutte's 4-flow conjecture states that every bridgeless
graph with no Petersen minor has a nowhere-zero 4-flow
\cite{Tutte1966}; see also \cite{WangZhangZhang2009}.  In loopless cubic graphs, a nowhere-zero flow over $\mathbb Z_2^2$ is equivalent to a proper 3-edge-colouring~\cite{ThomasThomson2000}, so the conjecture contains Tutte's edge-colouring conjecture as its cubic case \cite{RobertsonSeymourThomas1997}. For every connected bridgeless plane graph $G$, planar duality shows that $G$ admits a nowhere-zero $4$-flow if and only if its dual $G^*$ admits a proper vertex $4$-colouring~\cite{WangZhangZhang2009}; hence the planar case follows from the Four-Colour Theorem~\cite{RobertsonEtAl1997FCT}.

Let $e$ be an edge of the Petersen graph $P$.  Kilakos and Shepherd proved that every 2-edge-connected cubic graph with no $P-e$ minor is 3-edge-colourable~\cite{KilakosShepherd1996}.  Thomas and Thomson extended this result to flows.  Applied componentwise, their theorem gives a nowhere-zero $\mathbb Z_2^2$-flow in every finite bridgeless $P-e$-minor-free graph~\cite{ThomasThomson2000}.

There is a second sequence of proper Petersen minors.  For $i=1,2,3$, let $P^{(i)}$ be obtained by contracting $i$ edges of a perfect matching of $P$, and put
\[
 Q=P^{(1)}=P/e.
\]
The isomorphism type is independent of the matching and of the selected edges~\cite{WangZhangZhang2009}. The graph $Q$ has nine vertices and fourteen edges; one vertex has degree four and the other eight have degree three. Ferguson gave a complete description of the 3-connected $P^{(3)}$-minor-free graphs and an asymptotic structure theorem for nonplanar quasi $4$-connected $P^{(2)}$-minor-free graphs \cite{Ferguson2015}.
Wang, Zhang and Zhang proved that every bridgeless $P^{(3)}$-minor-free graph has a nowhere-zero $4$-flow
\cite{WangZhangZhang2009}. 

The proper-minor chain $P^{(3)}\minor P^{(2)}\minor Q\minor P$ gives
\[
 \Ex(P^{(3)})\subsetneq\Ex(P^{(2)})\subsetneq\Ex(Q).
\]
Both inclusions are strict: $P^{(3)}$ witnesses the first and $P^{(2)}$ the second.  Our main result reaches the largest of these three excluded-minor classes and, in particular, settles the intermediate $P^{(2)}$-minor-free case.

\begin{theorem}\label{thm:main}
Every finite bridgeless undirected multigraph with no $Q$ minor admits a nowhere-zero $\mathbb Z_2^2$-flow, and hence a nowhere-zero 4-flow.
\end{theorem}

Up to isomorphism, the two maximal proper Petersen minors are $P-e$ and $Q=P/e$.  Indeed, in a minor sequence from $P$ to a proper minor $R$, the first operation is an edge deletion or contraction, or a vertex deletion that may be preceded by deleting an incident edge.  Since $P$ is edge-transitive, $R\minor P-e$ or $R\minor Q$.  The two excluded-minor classes are incomparable: $P-e\not\minor Q$ by order, while $Q\not\minor P-e$ because both graphs have fourteen edges and every operation reducing $P-e$ from ten vertices to nine also lowers the edge count.  Theorem~\ref{thm:main} and the theorem of Thomas and Thomson therefore imply the following.

\begin{corollary}\label{cor:proper}
For every proper minor $R$ of the Petersen graph, every finite bridgeless $R$-minor-free graph admits a nowhere-zero $4$-flow.  Equivalently, every finite bridgeless graph without a nowhere-zero $4$-flow contains every proper minor of $P$, and in particular contains both $P-e$ and $P/e$.
\end{corollary}

The proof follows the architecture of Thomas and Thomson's $P-e$-minor theorem~\cite{ThomasThomson2000}.  Their girth-five theorem leaves the planar dodecahedron as one of four base minors.  In an almost $4$-connected nonplanar host, the theorem of Norin and Thomas~\cite{NorinThomas2016} supplies a minor that is either a jump extension or a cross extension of a dodecahedral subdivision.  The new finite ingredient below shows explicitly that every such extension contains $Q$, as do the other three base minors.  Thomas and Thomson's minimal-obstruction lemmas then reduce the flow theorem to this structural statement and the planar case.

\section{Preliminaries and structural results}\label{sec:preliminaries}

All graphs are finite undirected multigraphs and may have loops, parallel edges, isolated vertices, or several components.  A multigraph minor is obtained by deleting vertices or edges and by contracting non-loop edges. Contraction retains multiplicities; loops created by contraction may subsequently be deleted.  When the target is simple, surplus parallel copies may be deleted while constructing a minor model. A loop is a cycle of length one, and two parallel edges form a cycle of length two. Thus a graph of girth at least five is simple.  A \emph{proper minor} of a graph is a minor not isomorphic to that graph.  An edge is a \emph{bridge} if its deletion increases the number of components.

Fix an orientation of a graph $G$ and set $\Gm=\mathbb Z_2\times\mathbb Z_2$.  A $\Gm$-flow is a map
$\varphi:E(G)\to\Gm$ satisfying
\[
 \sum_{f\in\delta^+(v)}\varphi(f)-
 \sum_{f\in\delta^-(v)}\varphi(f)=0
 \qquad (v\in V(G)).
\]
An oriented loop is counted once in each of $\delta^+(v)$ and $\delta^-(v)$.  The flow is nowhere-zero if $\varphi(f)\ne0$ for every edge $f$.  A graph has a nowhere-zero $\mathbb Z_2^2$-flow if and only if it has a nowhere-zero integer $4$-flow~\cite{ThomasThomson2000}.

For a simple graph $H$, an $H$-minor model in $G$ is a family $(B_x:x\in V(H))$ of nonempty, pairwise disjoint, connected vertex sets such that an edge of $G$ joins $B_x$ and $B_y$ whenever $xy\in E(H)$.  We write $H\minor G$ when such a model exists, and $\Ex(H)=\set{G:H\not\minor G}$.  A \emph{separation} of a graph $G$ is a pair $(A,B)$ of vertex sets such that $A\cup B=V(G)$ and no edge joins $A\setminus B$ to $B\setminus A$.  Its order is $|A\cap B|$.  A graph is \emph{$3$-connected} if it has at least four vertices and remains connected after the deletion of any set of at most two vertices.

A simple graph is \emph{almost 4-connected} if it is 3-connected, has at least five vertices, and every separation $(A,B)$ of order three satisfies $|A|\le4$ or $|B|\le4$.  Thomas and Thomson define \emph{quasi 4-connectivity} by the same condition~\cite{ThomasThomson2000}. Norin and Thomas instead forbid a partition $V(G)=X\mathbin{\dot\cup}Y \mathbin{\dot\cup}Z$ with $|X|,|Y|\ge2$, $|Z|=3$, and no edge from $X$ to $Y$~\cite[corrected version v3]{NorinThomas2016}; such a partition is precisely the separation $(X\cup Z,Y\cup Z)$ excluded above, so all three notions coincide.

Let $H$ be a 3-connected planar graph with its spherical embedding.  A \emph{jump extension} of $H$ is $H+uv$, where no facial cycle contains both $u$ and $v$.  A \emph{cross extension} is $H+u_1v_1+u_2v_2$, where $u_1,u_2,v_1,v_2$ are distinct and occur in this cyclic order on a facial cycle.

The first structural ingredient is Theorem 1.6 of Thomas and Thomson \cite{ThomasThomson2000}.

\begin{theorem}[Thomas and Thomson]\label{thm:TT-base} Every graph with minimum degree at least three and no cycle of length less than five has a minor isomorphic to Triplex, Petersen, Dodecahedron, or Basket.
\end{theorem}

We use statement (1.2) from Norin and Thomas \cite[(1.2), corrected version v3, p.~4; proved there as~(6.6)]{NorinThomas2016}. In our notation, their statement reads as follows.

\begin{theorem}[Norin and Thomas]\label{thm:NT}
Let $H$ be an almost 4-connected triangle-free planar graph, and let $G$ be an almost 4-connected nonplanar graph containing a subdivision of $H$.  Then $G$ has as a minor either a jump extension or a cross extension of $H$.
\end{theorem}

We also use a standard equivalence for subcubic targets.

\begin{lemma}\label{lem:subcubic}
If $H$ is a simple graph with maximum degree at most three and $H\minor G$, then $G$ contains a subdivision of $H$ as a subgraph.
\end{lemma}

\begin{proof}
Fix an $H$-minor model and one host edge for every adjacency of $H$. If $x\in V(H)$ is isolated, retain any one vertex of $B_x$. Otherwise, within $B_x$ choose a minimal tree joining the attachment vertices used by the edges incident with $x$.  There are at most three attachments, so this tree is the union of internally disjoint paths from a single vertex to those attachments, with paths of length zero allowed.  Retaining these trees and the selected edges between branch sets gives a subdivision of $H$.
\end{proof}

\section{Petersen contractions in the girth-five structure}\label{sec:finite}

Let $T$, $D$, and $B$ denote the Triplex, dodecahedral, and Basket graphs in Theorem~\ref{thm:TT-base}.  Figure~\ref{fig:labelled-graphs} fixes the labels used below.  Appendix~\ref{app:certificates} gives their edge descriptions and the label crosswalk to the drawings of Thomas and Thomson.

\begin{figure}[pos=htbp]
\centering
\tikzset{
  gedge/.style={draw=black,line width=.45pt,line cap=round},
  bridge/.style={gedge,preaction={draw=white,line width=1.6pt}},
  vtx/.style={circle,draw=black,fill=white,inner sep=0pt,
    minimum size=4.4mm,font=\scriptsize,line width=.45pt},
  vtx4/.style={vtx,line width=1.1pt}
}
\begin{tabular}{cc}
\begin{tikzpicture}
  \useasboundingbox (-2.52,-2.35) rectangle (2.52,2.35);
  \coordinate (q0) at (0,0);
  \coordinate (q1) at (-.10,.95);
  \coordinate (q2) at (-.75,-.55);
  \coordinate (q3) at (1.10,-.15);
  \coordinate (q4) at (-1.45,1.10);
  \coordinate (q5) at (.55,-1.55);
  \coordinate (q6) at (.45,1.75);
  \coordinate (q7) at (-1.45,-1.15);
  \coordinate (q8) at (1.80,.10);
  \foreach \u/\v in {0/1,0/3,0/4,0/5,1/6,2/7,3/8,
                     4/6,4/7,5/7,5/8,6/8}
    \draw[gedge] (q\u)--(q\v);
  \draw[bridge] (q1)--(q2);
  \draw[bridge] (q2)--(q3);
  \node[vtx4] at (q0) {0};
  \foreach \i in {1,2,3,4,5,6,7,8}
    \node[vtx] at (q\i) {\i};
  \node[font=\small] at (0,-2.05) {(a) $Q=P/e$};
\end{tikzpicture}
&
\begin{tikzpicture}
  \useasboundingbox (-2.52,-2.42) rectangle (2.52,2.42);
  \foreach \i in {0,...,8}
    \coordinate (t\i) at ({90-40*\i}:1.72cm);
  \coordinate (t9) at (90:.72cm);
  \coordinate (t10) at (210:.72cm);
  \coordinate (t11) at (330:.72cm);
  \foreach \u/\v in {0/1,1/2,2/3,3/4,4/5,5/6,6/7,7/8,8/0}
    \draw[gedge] (t\u)--(t\v);
  \foreach \v in {0,3,6}
    \draw[gedge] (t9)--(t\v);
  \foreach \v in {1,4,7}
    \draw[bridge] (t10)--(t\v);
  \foreach \v in {2,5,8}
    \draw[bridge] (t11)--(t\v);
  \foreach \i in {0,...,11}
    \node[vtx] at (t\i) {\i};
  \node[font=\small] at (0,-2.20) {(b) Triplex};
\end{tikzpicture}
\\[1.5em]
\begin{tikzpicture}[scale=0.8]
  \useasboundingbox (-3.2,-3.2) rectangle (3.35,3.35);

  \coordinate (b0)  at ( 0.00,  2.45);
  \coordinate (b1)  at (-1.25,  1.28);
  \coordinate (b2)  at ( 1.25,  1.28);
  \coordinate (b3)  at (-2.45,  0.00);
  \coordinate (b11) at ( 2.45,  0.00);
  \coordinate (b8)  at (-1.20, -1.32);
  \coordinate (b9)  at ( 1.20, -1.32);
  \coordinate (b10) at ( 0.00, -2.48);

  \coordinate (b4)  at (-0.58,  0.45);
  \coordinate (b5)  at ( 0.58,  0.45);
  \coordinate (b6)  at (-0.58, -0.45);
  \coordinate (b7)  at ( 0.58, -0.45);

  \foreach \u/\v in {0/1, 0/2,
    1/3, 3/8, 8/10,
    10/9, 9/11, 11/2,
    1/4, 2/5, 8/6, 9/7,
    4/5, 5/7, 7/6, 6/4}{\draw[gedge] (b\u)--(b\v);}

  \draw[gedge]
    (b3) .. controls (-2.85, 4.00) and ( 3.05, 4.00) .. (b11);

  \draw[bridge]
    (b0) .. controls ( 4.15, 2.20) and ( 4.15,-2.20) .. (b10);

  \foreach \i in {0,1,2,3,4,5,6,7,8,9,10,11}{
    \node[vtx] at (b\i) {\i};
  }

  \node[font=\small] at (0,-3) {(c) Basket};
\end{tikzpicture}
&
\begin{tikzpicture}
  \useasboundingbox (-2.55,-2.65) rectangle (2.55,2.45);
  \foreach \i/\a in {0/90,1/18,11/-54,19/-126,9/162}
    \coordinate (d\i) at (\a:2.25cm);
  \foreach \i/\a in {10/90,12/54,2/18,3/-18,13/-54,15/-90,
                     17/-126,7/-162,8/162,18/126}
    \coordinate (d\i) at (\a:1.45cm);
  \foreach \i/\a in {14/54,4/-18,5/-90,6/-162,16/126}
    \coordinate (d\i) at (\a:.72cm);
  \foreach \u/\v in {0/1,1/11,11/19,19/9,9/0,
                     10/12,12/2,2/3,3/13,13/15,15/17,17/7,7/8,8/18,18/10,
                     14/4,4/5,5/6,6/16,16/14,
                     0/10,1/2,11/13,19/17,9/8,
                     12/14,3/4,15/5,7/6,18/16}
    \draw[gedge] (d\u)--(d\v);
  \foreach \i in {0,...,19}
    \node[vtx] at (d\i) {\i};
  \node[font=\small] at (0,-2.48) {(d) Dodecahedron};
\end{tikzpicture}
\end{tabular}
\caption{The labelled graphs used in the finite argument. The degree-four vertex of $Q$ has a heavier outline.}
\label{fig:labelled-graphs}
\end{figure}
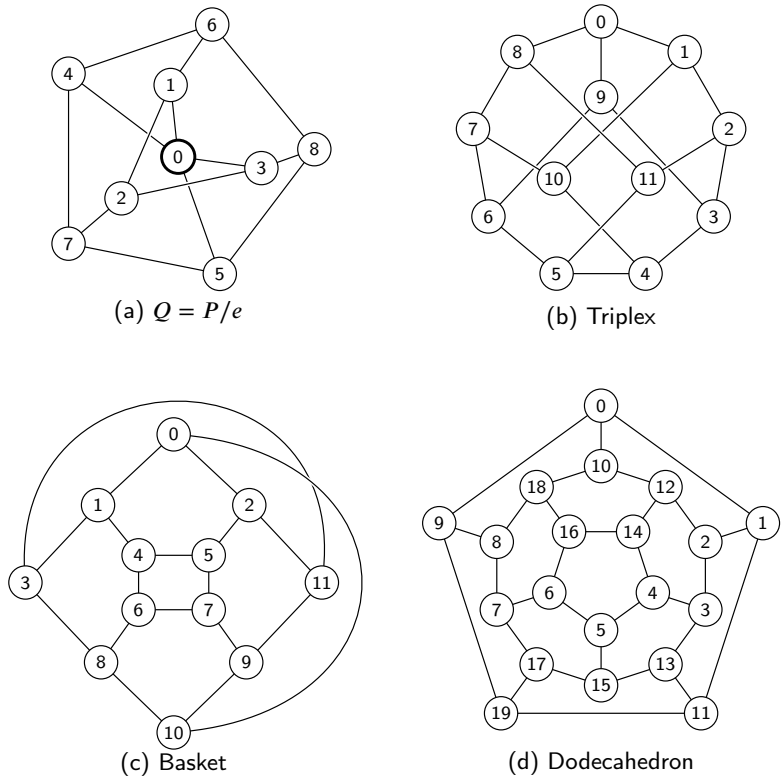

Thomas and Thomson note in the proof of their Theorem 3.2 that $D$ is quasi 4-connected~\cite{ThomasThomson2000}.  It is also planar, cubic, and triangle-free, and its standard embedding has twelve pentagonal faces.  Thus $D$ satisfies the base-graph hypotheses of Theorem~\ref{thm:NT}.

The extensions of $D$ fall into four orbits.  The following formulation makes
the finite reduction explicit.

\begin{lemma}\label{lem:orbits}
The noncofacial unordered vertex pairs of $D$ form three automorphism orbits, represented by $\{0,3\}$, $\{0,4\}$, and $\{0,5\}$.  The facial crosses form one orbit, represented by the two chords $0\!-\!2$ and $1\!-\!12$ of the face $(0,1,2,12,10)$.
\end{lemma}

\begin{proof}
We use only three visible symmetries of the labelled dodecahedron.  From the edge description in Appendix~\ref{app:certificates}, direct inspection shows that the following permutations preserve $E(D)$:
\begin{align*}
 \rho={}&(0\ 1\ \cdots\ 9)(10\ 11\ \cdots\ 19),\\
 \alpha={}&(1\ 10\ 9)(2\ 18\ 19)(3\ 16\ 17)
              (4\ 6\ 15)(7\ 13\ 14)(8\ 11\ 12),\\
 \tau={}&(1\ 10)(2\ 12)(3\ 14)(6\ 15)(7\ 17)
             (8\ 19)(11\ 18)(13\ 16).
\end{align*}
Geometrically, $\alpha$ is the rotation of order three about the axis through the antipodal vertices $0$ and $5$, while $\tau$ is a reflection fixing that axis.

The permutation $\rho$ is transitive on each of $\{0,\ldots,9\}$ and $\{10,\ldots,19\}$, and $\tau$ sends $1$ to $10$.  Hence the subgroup generated by $\rho$ and $\tau$ is vertex-transitive.  The distance spheres about $0$, read directly from the edge description, are
\begin{align*}
 \Gamma_1(0)&=\{1,9,10\},\\
 \Gamma_2(0)&=\{2,8,11,12,18,19\},\\
 \Gamma_3(0)&=\{3,7,13,14,16,17\},\\
 \Gamma_4(0)&=\{4,6,15\},\\
 \Gamma_5(0)&=\{5\}.
\end{align*}
Both $\alpha$ and $\tau$ fix $0$.  The permutation $\alpha$ cycles $\Gamma_1(0)$ and $\Gamma_4(0)$.  On $\Gamma_2(0)$ it has the two cycles
\[
 (2\ 18\ 19)\quad\text{and}\quad(8\ 11\ 12),
\]
which are joined by $\tau(2)=12$; on $\Gamma_3(0)$ it has the two cycles
\[
 (3\ 16\ 17)\quad\text{and}\quad(7\ 13\ 14),
\]
which are joined by $\tau(3)=14$.  Thus the stabilizer of $0$ is transitive on every distance sphere about $0$.  Together with vertex-transitivity, this shows that two unordered vertex pairs of $D$ are in the same automorphism orbit whenever their endpoints have the same distance.

Two vertices on a pentagonal face have distance at most two.  Conversely, adjacent vertices lie on a common face, and if $x-z-y$ is a path of length two, then the two edges $zx$ and $zy$ occur consecutively around the cubic vertex $z$ and hence lie on the boundary of a common face.  Therefore two vertices of $D$ are cofacial if and only if their distance is at most two.  The noncofacial pairs consequently have distances three, four, or five, giving the three orbits represented by $\{0,3\}$, $\{0,4\}$, and $\{0,5\}$.

It remains to consider crosses.  On a pentagonal face, a cross is uniquely determined by the omitted fifth vertex, and hence by an incident vertex--face pair.  The subgroup above is transitive on vertices.  Moreover, $\alpha$ fixes $0$ and cyclically permutes the three faces incident with $0$:
\begin{align*}
 &(0,1,2,12,10),\\
 &(0,9,8,18,10),\\
 &(0,1,11,19,9).
\end{align*}
Thus the automorphism group is transitive on incident vertex--face pairs, and therefore on facial crosses.  Taking the omitted vertex to be $10$ in the first displayed face gives the representative chords $0\!-\!2$ and $1\!-\!12$.
\end{proof}

The three jump representatives are pairwise nonisomorphic.  Indeed, the added edge is the unique edge whose ends both have degree four.  Deleting it recovers $D$, and the distance between its ends is therefore an isomorphism invariant.

\begin{lemma}\label{lem:finite}
Each of $P$, $T$, and $B$ contains $Q$ as a minor.  Every jump extension and every cross extension of $D$ contains $Q$ as a minor.
\end{lemma}

\begin{proof}
The labelled host graphs are defined in Appendix~\ref{app:certificates}.  Tables~\ref{tab:base-models} and~\ref{tab:d-models} specify the deletions and the nine branch sets for each of the seven representative hosts.  An entry $\pbranch{x_1,\ldots,x_k}$ means that $x_1\cdots x_k$ is a path in the graph remaining after the indicated deletions; hence every displayed branch set is connected, and disjointness is immediate from the tables.  Tables~\ref{tab:base-witnesses} and~\ref{tab:d-witnesses} then give one surviving host edge for each of the fourteen edges of $Q$.  Contracting the displayed paths and deleting all surplus edges therefore leaves exactly $Q$.

For the dodecahedron, Lemma~\ref{lem:orbits} shows that the three jump hosts and the one cross host in Table~\ref{tab:d-models} represent all possible extensions.  The displayed models consequently prove the assertion for every jump and cross extension of $D$.
\end{proof}

The appendix thus gives a fully explicit deletion-and-contraction proof of the finite lemma.

The preceding finite lemma resolves the only case in Theorem~\ref{thm:TT-base} in which the required Petersen contraction is not already present in the base graph. Combining it with the structural results of Section~\ref{sec:preliminaries} yields the following strengthening of the girth-five theorem of Thomas and Thomson.

\begin{theorem}\label{thm:structural}
Every almost \(4\)-connected nonplanar graph with minimum degree at least three and girth at least five contains \(Q\) as a minor.
\end{theorem}

\begin{proof}
Let \(G\) satisfy the hypotheses. By Theorem~\ref{thm:TT-base}, one of \(T,P,D,B\) is a minor of \(G\). If \(G\) contains \(T\), \(P\), or \(B\), then Lemma~\ref{lem:finite} gives \(Q\minor G\).

It remains to consider the case \(D\minor G\). Since \(D\) is cubic, Lemma~\ref{lem:subcubic} gives a subdivision of \(D\) in \(G\). The graph \(D\) is almost \(4\)-connected, planar, and triangle-free, so Theorem~\ref{thm:NT} supplies a jump or cross extension of \(D\) as a minor of \(G\). By Lemma~\ref{lem:finite}, every such extension contains \(Q\) as a minor. Thus \(Q\minor G\) in every case.
\end{proof}

\section{From structure to flows}\label{sec:flows}

Call a graph $H$ \emph{flow-minimal} if it has no bridge and no nowhere-zero $\mathbb Z_2^2$-flow, while every proper bridgeless minor of $H$ has such a flow.  Equivalently, every proper minor either has a bridge or has a nowhere-zero $\mathbb Z_2^2$-flow. This is the minimality condition used in \cite{ThomasThomson2000}.

A flow-minimal graph is connected.  Otherwise, one of its components has no nowhere-zero flow, and that component is a proper bridgeless minor.  It is also loopless: deleting a loop preserves bridgelessness, and any nowhere-zero flow of the deletion extends by assigning an arbitrary nonzero value to the loop. Thus either deletion outcome contradicts flow-minimality.

The following is the form in which we use the minimal-obstruction lemmas of Thomas and Thomson~\cite[Lemmas~4.1--4.4]{ThomasThomson2000}.

\begin{lemma}\label{lem:minimal}
Every flow-minimal multigraph is \(3\)-connected, has minimum degree at least three and girth at least five, and is almost \(4\)-connected.
\end{lemma}

The girth conclusion makes the graph simple, and quasi and almost $4$-connectivity agree by Section~\ref{sec:preliminaries}.

We now combine the structural theorem with the minimal-obstruction lemma to prove the main flow result.

\begin{proof}[Proof of Theorem~\ref{thm:main}]
Assigning an arbitrary nonzero value to a loop does not affect the flow equation, so delete all loops and restore them at the end.  Suppose that the remaining graph has no nowhere-zero $\Gm$-flow.  Since flows are componentwise, some component $C$ has no such flow.  Every component of a bridgeless graph is bridgeless, so $C$ is bridgeless and $Q$-minor-free.  Among the minors of $C$ that have no bridge and no nowhere-zero $\Gm$-flow, choose $H$ minimizing $|V(H)|+|E(H)|$; this class is nonempty because it contains $C$.  Every proper minor of $H$ has smaller $|V|+|E|$, so $H$ is flow-minimal, and minor-closedness gives $Q\not\minor H$.

Lemma~\ref{lem:minimal} shows that $H$ has minimum degree at least three, girth at least five, and is almost 4-connected.  If $H$ is planar, the planar duality form of the Four-Color Theorem~\cite{RobertsonEtAl1997FCT} (see also~\cite{WangZhangZhang2009}) gives a nowhere-zero $\Gm$-flow, a contradiction.  If $H$ is nonplanar, Theorem~\ref{thm:structural} gives $Q\minor H$, again a contradiction.  Hence every component of the loopless graph has a nowhere-zero $\Gm$-flow.  Their union, together with arbitrary nonzero values on the deleted loops, is the required flow on $G$.  Thus \(G\) has a nowhere-zero \(\mathbb Z_2^2\)-flow, and hence a nowhere-zero \(4\)-flow.
\end{proof}

\begin{proof}[Proof of Corollary~\ref{cor:proper}]
Let $R$ be a proper minor of $P$.  As observed above, $R\minor P-e$ or $R\minor P/e$; hence an $R$-minor-free graph is $P-e$-minor-free or $P/e$-minor-free, respectively.  Apply the theorem of Thomas and Thomson~\cite{ThomasThomson2000} or Theorem~\ref{thm:main}.  Taking contrapositives gives the equivalent statement, and disconnected graphs are covered because flows are componentwise.
\end{proof}

\appendix
\setcounter{table}{0}
\renewcommand{\thetable}{A.\arabic{table}}
\renewcommand{\theHtable}{appendix.\arabic{table}}
\section{Labelled graphs and explicit minor models}\label{app:certificates}

Label the Petersen graph with outer cycle $0,1,2,3,4$, inner star $5,7,9,6,8$, and spokes $i(5+i)$.  Contracting $01$, labelling the resulting vertex $0$, and replacing each old label $j\in\{2,\ldots,9\}$ by $j-1$ gives
\[
 E(Q)=\{01,03,04,05,12,16,23,27,38,46,47,57,58,68\}.
\]
The Triplex $T$ consists of the $9$-cycle $0,1,\ldots,8$ and three further vertices $9+r$, each adjacent to $r,r+3,r+6$, for $r=0,1,2$.  In the Basket $B$, the outer and inner cycles are
\[
 (0,1,3,8,10,9,11,2)\qquad\text{and}\qquad(4,5,7,6),
\]
with attachments $14,25,86,97$ and the remaining edges $0\,10$ and $3\,11$.  These are the labellings shown in Figure~\ref{fig:labelled-graphs}; comparison of the displayed cycles and attachments with Thomas and Thomson~\cite[Figure~1]{ThomasThomson2000} gives explicit isomorphisms to their Triplex and Basket.

Write the dodecahedron as $D=GP(10,2)$ on vertices $0,\ldots,19$.  With indices modulo $10$, its edges are
\[
 i(i+1),\qquad i(10+i),\qquad (10+i)(10+(i+2))
 \qquad(0\le i<10).
\]
This is the labelling in Figure~\ref{fig:labelled-graphs}; in particular, $(0,1,2,12,10)$ is a facial pentagon.

In the next two tables, the branch sets are listed in the order $0,\ldots,8$ of the vertices of $Q$.  A singleton $x$ denotes $\{x\}$, while $\pbranch{x_1,\ldots,x_k}$ denotes the vertex set of the path $x_1\cdots x_k$.  The deletion column is applied before the paths are contracted.

\par\medskip
\noindent\begin{minipage}{\textwidth}

\centering
\footnotesize
\renewcommand{\arraystretch}{1.18}
\begin{tabularx}{\textwidth}{@{}K{.09\textwidth}K{.13\textwidth}Y@{}}
\toprule
Host & Delete & Branch sets for vertices $0,\ldots,8$ of $Q$ \\
\midrule
$P$ & none & $\pbranch{0,1}\bsep2\bsep3\bsep4\bsep5\bsep6\bsep7\bsep8\bsep9$ \\
$T$ & edge $(1,2)$ & $\pbranch{4,10}\bsep\pbranch{0,1}\bsep8\bsep7\bsep\pbranch{2,3}\bsep5\bsep9\bsep11\bsep6$ \\
$B$ & edge $(5,7)$ & $\pbranch{0,10}\bsep1\bsep3\bsep8\bsep2\bsep\pbranch{7,9}\bsep\pbranch{4,5}\bsep11\bsep6$ \\
\bottomrule
\end{tabularx}
\captionof{table}{Explicit minor models in the three base graphs.}
\label{tab:base-models}
\end{minipage}
\par\medskip

\par\medskip
\noindent\begin{minipage}{\textwidth}

\centering
\footnotesize
\renewcommand{\arraystretch}{1.20}
\begin{tabularx}{\textwidth}{@{}K{.14\textwidth}K{.21\textwidth}Y@{}}
\toprule
Host & Delete & Branch sets for vertices $0,\ldots,8$ of $Q$ \\
\midrule
$D+(0,3)$ & vertices $7,17$; edges $(0,1),(2,3),(14,16)$ & $\pbranch{15,5,6,16,18,8}\bsep13\bsep11\bsep\pbranch{9,19}\bsep4\bsep10\bsep3\bsep\pbranch{1,2,12,14}\bsep0$ \\
$D+(0,4)$ & vertices $8,10,18$; edges $(0,1),(4,14)$ & $\pbranch{1,2,12,14,16,6,7}\bsep3\bsep\pbranch{9,0,4}\bsep5\bsep11\bsep17\bsep13\bsep19\bsep15$ \\
$D+(0,5)$ & vertices $8,9,10,18$; edge $(5,6)$ & $\pbranch{14,16,6,7,17,19}\bsep11\bsep\pbranch{0,1}\bsep\pbranch{2,12}\bsep15\bsep4\bsep13\bsep5\bsep3$ \\
$D+(0,2)+(1,12)$ & vertices $13,15$; edges $(0,1),(1,2),(9,19),(2,12)$ & $\pbranch{9,0,2,3,4,5}\bsep6\bsep\pbranch{1,11,19,17,7}\bsep8\bsep14\bsep10\bsep16\bsep12\bsep18$ \\
\bottomrule
\end{tabularx}
\captionof{table}{Explicit minor models in the four dodecahedral orbit representatives.  Paths may use the added edges of their host; in the last row, $(0,2)$ and $(1,12)$ are the two added chords of the face $(0,1,2,12,10)$.}
\label{tab:d-models}
\end{minipage}
\par\medskip

The next two tables record one surviving host edge for every edge of $Q$.  Columns are labelled by the edges of $Q$ and each table is split into two blocks.  Together with Tables~\ref{tab:base-models} and~\ref{tab:d-models}, they give a line-by-line verification of all seven models.

\par\medskip
\noindent\begin{minipage}{\textwidth}
\centering
\footnotesize
\renewcommand{\arraystretch}{1.10}
\begin{tabular}{@{}lccccccc@{}}
\toprule
Host & $01$ & $03$ & $04$ & $05$ & $12$ & $16$ & $23$ \\
\midrule
$P$ & $(1,2)$ & $(0,4)$ & $(0,5)$ & $(1,6)$ & $(2,3)$ & $(2,7)$ & $(3,4)$ \\
$T$ & $(1,10)$ & $(7,10)$ & $(3,4)$ & $(4,5)$ & $(0,8)$ & $(0,9)$ & $(7,8)$ \\
$B$ & $(0,1)$ & $(8,10)$ & $(0,2)$ & $(9,10)$ & $(1,3)$ & $(1,4)$ & $(3,8)$ \\
\bottomrule
\end{tabular}

\smallskip
\begin{tabular}{@{}lccccccc@{}}
\toprule
Host & $27$ & $38$ & $46$ & $47$ & $57$ & $58$ & $68$ \\
\midrule
$P$ & $(3,8)$ & $(4,9)$ & $(5,7)$ & $(5,8)$ & $(6,8)$ & $(6,9)$ & $(7,9)$ \\
$T$ & $(8,11)$ & $(6,7)$ & $(3,9)$ & $(2,11)$ & $(5,11)$ & $(5,6)$ & $(6,9)$ \\
$B$ & $(3,11)$ & $(6,8)$ & $(2,5)$ & $(2,11)$ & $(9,11)$ & $(6,7)$ & $(4,6)$ \\
\bottomrule
\end{tabular}
\captionof{table}{Witness edges for the base-graph minor models.}
\label{tab:base-witnesses}
\end{minipage}
\par\medskip

For brevity, put $D_{03}=D+(0,3)$, $D_{04}=D+(0,4)$, $D_{05}=D+(0,5)$, and $D_{\times}=D+(0,2)+(1,12)$.

\par\smallskip
\noindent\begin{minipage}{\textwidth}
\centering
\footnotesize
\renewcommand{\arraystretch}{1.10}
\begin{tabular}{@{}lccccccc@{}}
\toprule
Host & $01$ & $03$ & $04$ & $05$ & $12$ & $16$ & $23$ \\
\midrule
$D_{03}$ & $(13,15)$ & $(8,9)$ & $(4,5)$ & $(10,18)$ & $(11,13)$ & $(3,13)$ & $(11,19)$ \\
$D_{04}$ & $(2,3)$ & $(5,6)$ & $(1,11)$ & $(7,17)$ & $(3,4)$ & $(3,13)$ & $(4,5)$ \\
$D_{05}$ & $(11,19)$ & $(12,14)$ & $(15,17)$ & $(4,14)$ & $(1,11)$ & $(11,13)$ & $(1,2)$ \\
$D_{\times}$ & $(5,6)$ & $(8,9)$ & $(4,14)$ & $(0,10)$ & $(6,7)$ & $(6,16)$ & $(7,8)$ \\
\bottomrule
\end{tabular}

\smallskip
\begin{tabular}{@{}lccccccc@{}}
\toprule
Host & $27$ & $38$ & $46$ & $47$ & $57$ & $58$ & $68$ \\
\midrule
$D_{03}$ & $(1,11)$ & $(0,9)$ & $(3,4)$ & $(4,14)$ & $(10,12)$ & $(0,10)$ & $(0,3)$ \\
$D_{04}$ & $(9,19)$ & $(5,15)$ & $(11,13)$ & $(11,19)$ & $(17,19)$ & $(15,17)$ & $(13,15)$ \\
$D_{05}$ & $(0,5)$ & $(2,3)$ & $(13,15)$ & $(5,15)$ & $(4,5)$ & $(3,4)$ & $(3,13)$ \\
$D_{\times}$ & $(1,12)$ & $(8,18)$ & $(14,16)$ & $(12,14)$ & $(10,12)$ & $(10,18)$ & $(16,18)$ \\
\bottomrule
\end{tabular}
\captionof{table}{Witness edges for the four dodecahedral models.  The entries $(0,3)$, $(0,5)$, and $(1,12)$ are added edges in the corresponding hosts.}
\label{tab:d-witnesses}
\end{minipage}
\par\medskip

\medskip
\noindent\textbf{Data availability.}
No datasets were generated or analysed in this study.

\medskip
\noindent\textbf{Funding.}
This work was supported by Project No.\ KDP-IKT-2023-900-I1-00000957/0000003, funded by the Ministry of Culture and Innovation of Hungary from the National Research, Development and Innovation Fund under the KDP-2023 scheme.  The funder had no role in the design, execution, interpretation, or reporting of the research.

\medskip
\noindent\textbf{Declaration of competing interest.}
The author declares no known competing financial interests or personal relationships that could have appeared to influence the work reported in this paper.

\medskip
\noindent\textbf{Declaration of generative AI and AI-assisted technologies in the manuscript preparation process.}
During the preparation of this work, the author used OpenAI's ChatGPT and Anthropic's Claude for brainstorming, literature search, \LaTeX{} drafting, and editorial assistance. The author reviewed and edited the output and take full responsibility for the content of the manuscript.

\setlength{\bibsep}{0pt}
\renewcommand{\bibfont}{\fontsize{8pt}{9pt}\selectfont}
\bibliographystyle{cas-model2-names}
\bibliography{references}

\end{document}